\documentclass[oneside]{amsart}
\usepackage[T1]{fontenc}
\usepackage{geometry}
\usepackage{mathtools}
\usepackage{amstext}
\usepackage{amsthm}
\usepackage{amssymb}

\makeatletter
\theoremstyle{plain}
\newtheorem{thm}{\protect\theoremname}
\theoremstyle{plain}
\newtheorem{lem}[thm]{\protect\lemmaname}
\theoremstyle{remark}
\newtheorem{rem}[thm]{\protect\remarkname}
\theoremstyle{definition}
\newtheorem{example}[thm]{\protect\examplename}
\theoremstyle{plain}
\newtheorem{cor}[thm]{\protect\corollaryname}

\usepackage{amsmath,amssymb,stmaryrd}
\usepackage{cmtiup}

\DeclareMathOperator{\Spec}{\mathrm{Spec}}
\newcommand{\Z}{\mathbb{Z}}

\makeatother

\providecommand{\corollaryname}{Corollary}
\providecommand{\examplename}{Example}
\providecommand{\lemmaname}{Lemma}
\providecommand{\remarkname}{Remark}
\providecommand{\theoremname}{Theorem}

\begin{document}
\address[Kentaro Mitsui]{Department of Mathematical Sciences, University of the Ryukyus, 1 Senbaru, Nishihara-cho, Okinawa 903-0213, Japan}
\email{mitsui@math.u-ryukyu.ac.jp}
\address[Nobuo Sato]{Department of Mathematics, National Taiwan University, No.\,1, Sec.\,4, Roosevelt Rd., Taipei 10617, Taiwan (R.O.C.)}
\email{nbsato@ntu.edu.tw}
\title{A criterion for $p$-closedness of derivations in dimension two}
\author{Kentaro Mitsui and Nobuo Sato}
\begin{abstract}
Jacobson developed a counterpart of Galois theory for purely inseparable
field extensions in positive characteristic. In his theory, a certain
type of derivations replace the role of the generators of Galois groups.
This article provides a convenient criterion for determining such
derivations in dimension two. We also present examples demonstrating
the efficiency of our criterion.
\end{abstract}

\maketitle
Consider a field $K$ of characteristic $p>0$. Let $K^{p}$ denote
the subfield of $p$-th powers in $K$, $\mathcal{E}_{K}$ the set
of finite index subfields of $K$ that contain $K^{p}$, and $\mathcal{D}_{K}$
the set of finite dimensional restricted Lie algebras over $K$. Jacobson
showed that the map $\mathcal{D}_{K}\to\mathcal{E}_{K}$ sending $A$
to the subfield of $A$-constants in $K$ is a bijection \cite[p.\,189]{key-3}.
In this sense, $\mathcal{D}_{K}$ functions as a counterpart of the
Galois groups of Galois extensions for $\mathcal{E}_{K}$, and the
subset of $\mathcal{E}_{K}$ obtained by restricting the indices to
at most $p$ corresponds to the set of all restricted Lie algebras
generated by one element as $K$-vector spaces. Here, the generators
$\partial$ are termed $\textit{p-closed derivations}$, characterized
by the existence of an $a\in K$ such that $\partial^{p}=a\partial$.
Specifically, nonzero $p$-closed derivations are akin to the generators
of the Galois groups for purely inseparable subfields of $K$ of index
$p$, and they play a fundamental role in studying purely inseparable
extensions.

On the other hand, $p$-closed derivations have been employed in the
theory of algebraic surfaces in positive characteristic to construct
purely inseparable quotients of degree $p$. For instance, the nonexistence
of nonzero regular vector fields on $K3$ surfaces has been proved
as an application \cite[\S6, Theorem 7]{key-7}. They are also utilized
in the explicit construction of elliptic surfaces (e.g., \cite{key-4}).
Let $K$ be the function field of an affine coordinate ring $R$ of
a nonsingular surface. The subring $R^{\partial}$ of $\partial$-constants
in $R$ plays a similar role as the invariant ring under the action
of a finite group, and the embedding $R^{\partial}\to R$ induces
a purely inseparable quotient morphism $\text{Spec}\,R\to\text{Spec}\,R^{\partial}$
of degree at most $p$. The quotient $\text{Spec}\,R^{\partial}$
generally exhibits singularities and has been applied particularly
to the studies on non-taut rational double points unique to positive
characteristic (e.g., \cite{key-6}, \cite{key-5}).

This short article presents a concise and effective method for determining
the $p$-closedness of a derivation $\partial=f\partial_{x}+g\partial_{y}$
in the two-dimensional case. This method is valid for an arbitrary
$\partial$, but it works most efficiently when $\partial_{x}(f)+\partial_{y}(g)=0$.
Such derivations appear in the studies on the aforementioned singularities.
Furthermore, the method indicates the extent to which $\partial$
is $p$-closed. More precisely, we provide a clean and computationally
powerful formula for $\partial(x)\partial^{p}(y)-\partial(y)\partial^{p}(x)$,
whose vanishing is equivalent to $\partial$ being $p$-closed. In
the general case, it is necessary to first find an $a\in K^{\times}$
such that $\partial_{x}(af)+\partial_{y}(ag)=0$, and a nontrivial
solution to a system of $p^{2}-1$ linear homogeneous equations in
$p^{2}$ variables provides such an $a$.

The proof of the main theorem utilizes the Cartier operator. In positive
characteristic algebraic geometry, the Cartier isomorphism is considered
pivotal (e.g., \cite{key-2}), and while the Cartier operator has
a well-known definition on closed forms, our proof revisits its definition
on all $1$-forms \cite{key-1}, applying the fact that the two definitions
agree.
\begin{lem}
\label{lem:p-closed}Suppose that $L$ is a subfield of a field $K$
of characteristic $p>0$ and that $K$ has a $p$-basis $x_{1},\ldots,x_{n}$
over $L$. Then, for given $n$-elements $f_{1},\ldots,f_{n}$ of
$K$, there exists $a\in K^{\times}$ such that $\sum_{i=1}^{n}\partial_{x_{i}}(af_{i})=0$.
\end{lem}

\begin{rem}
The proof below is constructive and gives an algorithm for finding
the multiplier $a$ of Lemma \ref{lem:p-closed}.
\end{rem}

\begin{proof}
Let $g\coloneqq\sum_{i=1}^{n}\partial_{x_{i}}\left(af_{i}\right)$.
Notice that $b\in K$ has a unique expression as $b=\sum_{I\in[0,p-1]^{n}}b_{I}x^{I}$
($b_{I}\in K^{p}L$), where $[0,p-1]\coloneqq\left\{ i\in\mathbb{Z}\mid0\leq i\leq p-1\right\} $
and $K^{p}L\subset K$. Thus, $a\in K^{\times}$ satisfying $g=0$
is given as a non-trivial solution to the system $S$ of homogeneous
linear equations for $p^{n}$ unknowns $(a_{I})_{I\in[0,p-1]^{n}}$.
The rank of the coefficient matrix of $S$ then must be at most $p^{n}-1$
since $g_{I}=0$ for $I$ whose entries are all $p-1$, implying the
existence of a non-trivial solution for $S$. Hence, the desired $a$
exists.
\end{proof}
\begin{thm}
\label{thm:p-closed}Let $K$ be a field of characteristic $p>0$.
Pick two elements $f,g$ of $K$ and a subfield $L$. Suppose that
$K$ has a $p$-basis $x,y$ over $L$. Let $a\in K^{\times}$ be
an element such that $\partial_{x}(af)+\partial_{y}(ag)=0$ (whose
existence is guaranteed by Lemma \ref{lem:p-closed}). Then, for $\partial\coloneqq f\partial_{x}+g\partial_{y}$,
it holds that 
\[
a\left(\partial(x)\partial^{p}(y)-\partial(y)\partial^{p}(x)\right)=f^{p}\partial_{x}^{p-1}(ag)-g^{p}\partial_{y}^{p-1}(af).
\]
\end{thm}

\begin{rem}
\label{rem:p-closed}Unlike the left-hand side, the right-hand side
of Theorem \ref{thm:p-closed} can be calculated easily. For $b=\sum_{(i,j)\in[0,p-1]^{2}}b_{i,j}x^{i}y^{j}$
($b_{i,j}\in K^{p}L$), we have $\partial_{x}^{p-1}(b)=-\sum_{j=0}^{p-1}b_{p-1,j}y^{j}$
and $\partial_{y}^{p-1}(b)=-\sum_{i=0}^{p-1}b_{i,p-1}x^{i}$ since
$(p-1)!\equiv-1\bmod{p}$. Hence, $\partial_{x}^{p-1}(ag)$ (resp.\ $\partial_{y}^{p-1}(af)$)
equals the coefficient of $-ag$ (resp.\ $-af$) at $x^{p-1}$ (resp.\ $y^{p-1}$)
since $\partial_{x}(af)+\partial_{y}(ag)=0$.
\end{rem}

\begin{proof}
Let $\omega\coloneqq ag\,dx-af\,dy$. From $\omega(\partial)=0$ and
the definition of the Cartier operator (see \cite[Ch.\,2, \S6, p.\,200]{key-1}),
it follows that $a\left(\partial(x)\partial^{p}(y)-\partial(y)\partial^{p}(x)\right)=-\omega(\partial^{p})=-(C\omega(\partial))^{p}$.
Since $d\omega=0$, we find $C\omega=-\partial_{x}^{p-1}(ag)^{1/p}dx+\partial_{y}^{p-1}(af)^{1/p}dy$
by the formula in the last paragraph of \cite[the proof of Proposition 8, p.\,202]{key-1}
and Remark \ref{rem:p-closed}. Hence, $-\left(C\omega(\partial)\right)^{p}=f^{p}\partial_{x}^{p-1}(ag)-g^{p}\partial_{y}^{p-1}(af).$
This completes the proof.
\end{proof}
\begin{example}
\label{exa:p-closed}Suppose that $\partial_{x}(f)+\partial_{y}(g)=0$.
We let $c_{f}\coloneqq-\partial_{y}^{p-1}(f)$ and $c_{g}\coloneqq-\partial_{x}^{p-1}(g)$.
By definition, $c_{f},c_{g}\in K^{p}L$, and there exists $h\in K$
such that $f=\partial_{y}(h)+c_{f}y^{p-1}$ and $g=-\partial_{x}(h)+c_{g}x^{p-1}$.
Theorem \ref{thm:p-closed} implies
\[
\partial\text{ is }p\text{-closed }\Leftrightarrow\text{ }f^{p}c_{g}=g^{p}c_{f}\text{ }\Leftrightarrow\text{ }\exists c\in K^{p}L,\text{ }(c_{f},c_{g})=(cf^{p},cg^{p}).\tag{{\ensuremath{*}}}
\]
Thus, in particular, if either $f=0$ or $g=0$, then $\partial$
is $p$-closed. We note the following two special cases:
\begin{enumerate}
\item Suppose that $K$ contains the polynomial ring $k[x,y]$ in two indeterminates
$x,y$ over a subfield $k$ of $K^{p}L$ and that $f$ and $g$ are
nonzero coprime elements of $k[x,y]$. In that case, the first of
the equivalences ($*$) simplifies to
\[
\partial\text{ is }p\text{-closed }\Leftrightarrow\text{ }c_{f}=c_{g}=0.
\]
Notice here that the condition of $f$ and $g$ being coprime is crucial.
For example, for $f=g=(x-y)^{p-1}$, $\partial$ is obviously $p$-closed
although $c_{f}=c_{g}=1$.
\item Despite (1), there are plenty of nontrivial examples of $\partial$
for which both $c_{f}$ and $c_{g}$ are nonzero. Suppose that $K$
contains the formal power series ring $k[\![x,y]\!]$ in two indeterminates
$x,y$ over a subfield $k$ of $K^{p}L$ and let $f,g\in k[\![x,y]\!]$.
Then, the equivalences ($*$) and straightforward calculations show
that the following conditions are equivalent:
\begin{enumerate}
\item $\partial$ is $p$-closed, and both $c_{f}\in f^{p}k[\![x,y]\!]$
and $c_{g}\in g^{p}k[\![x,y]\!]$ hold;
\item $f$ and $g$ take the forms
\begin{align*}
f & =\sum_{i\geq0}\partial_{y}(h)^{p^{i}}c^{\frac{p^{i}-1}{p-1}}y^{p^{i}-1}\\
g & =-\sum_{i\geq0}\partial_{x}(h)^{p^{i}}c^{\frac{p^{i}-1}{p-1}}x^{p^{i}-1}
\end{align*}
with some $h\in k[\![x,y]\!]$ and some $c\in k[\![x^{p},y^{p}]\!]$,
where the assumptions $\partial_{x}(f)+\partial_{y}(g)=0$ and $f,g\in k[\![x,y]\!]$
automatically hold.
\end{enumerate}
\end{enumerate}
\end{example}

\begin{cor}
\label{cor:monomial}Take two elements\textup{ $m_{x},m_{y}$ of the
localization $\Z_{(p)}$ of $\mathbb{Z}$ at the prime ideal $(p)$}.
Set $(n_{x},n_{y})\coloneqq(m_{x}+1,m_{y}+1)$. Then, we have
\[
y^{m_{y}}\partial_{x}+x^{m_{x}}\partial_{y}\text{ is }p\text{-closed }\Leftrightarrow\text{ }n_{x},n_{y}\in\Z_{(p)}^{\times}\text{ or }n_{x}=n_{y}=0.
\]
\end{cor}

\begin{proof}
First, notice that $(a,f,g)\coloneqq(1,y^{m_{y}},x^{m_{x}})$ clearly
satisfies the assumption $\partial_{x}(af)+\partial_{y}(ag)=0$ of
Theorem \ref{thm:p-closed}. Then, Remark \ref{rem:p-closed} gives
\[
f^{p}\partial_{x}^{p-1}(g)-g^{p}\partial_{y}^{p-1}(f)=(xy)^{-p}\left(\varepsilon_{x}x^{n_{x}}y^{pn_{y}}-\varepsilon_{y}x^{pn_{x}}y^{n_{y}}\right),
\]
where
\[
\varepsilon_{z}\coloneqq\begin{cases}
0 & n_{z}\not\equiv0\bmod p\\
-1 & n_{z}\equiv0\bmod p.
\end{cases}
\]
Thus, Theorem \ref{thm:p-closed} implies
\[
y^{m_{y}}\partial_{x}+x^{m_{x}}\partial_{y}\text{ is }p\text{-closed }\Leftrightarrow\text{ }\varepsilon_{x}x^{n_{x}}y^{pn_{y}}=\varepsilon_{y}x^{pn_{x}}y^{n_{y}}.
\]
The latter condition can be divided into the two cases
\begin{enumerate}
\item $\varepsilon_{x}=\varepsilon_{y}=0$;
\item $\varepsilon_{x}=\varepsilon_{y}=-1$ and $(n_{x},pn_{y})=(pn_{x},n_{y})$,
\end{enumerate}
which are equivalent to 
\begin{enumerate}
\item[(1')] $n_{x},n_{y}\in\Z_{(p)}^{\times}$;
\item[(2')] $n_{x}=n_{y}=0$,
\end{enumerate}
respectively. This gives the claim.
\end{proof}
\begin{example}
Let $k$ be an algebraically closed field of characteristic $p=5$.
By (1) of Example \ref{exa:p-closed} (or more directly by Corollary
\ref{cor:monomial}), the derivation $\partial\coloneqq y\partial_{x}+x^{2}\partial_{y}$
is $p$-closed. The singularity of the resulting quotient $\Spec k[x,y]^{\partial}$
of the affine plane $\Spec k[x,y]$ is a non-taut rational double
point $E_{8}^{0}$ \cite[Part II, 3.1.3]{key-6}.
\end{example}

\subsection*{Acknowledgment}

This work was supported by JSPS KAKENHI Grant Number JP21K03179 and
MOST Grant Number 111-2115-M-002-003-MY3.

\end{document}